\documentclass
{elsarticle}

\usepackage{graphicx}
\usepackage{latexsym,amsmath}
\usepackage{amsmath,amsthm,amscd}
\usepackage{amsfonts}
\usepackage[psamsfonts]{amssymb}
\usepackage{enumerate}
\usepackage{mathtools}

\usepackage[final]
{showkeys}

\usepackage{url}

\usepackage{hyperref}

\theoremstyle{plain} 
\newtheorem{theorem}{Theorem}

{Lemma}

\newtheorem{proposition}[theorem]{Proposition}

\theoremstyle{definition} 

\theoremstyle{definition} 

\newtheorem*{ex*}{Example}
\theoremstyle{remark} 

\theoremstyle{remark} 

\newtheorem*{remark*}{Remark}

\newcommand{\al}{\alpha}

\newcommand{\si}{\sigma}

\newcommand{\la}{\lambda}

\newcommand{\ep}{\varepsilon}

\newcommand{\La}{\Lambda}

\renewcommand{\P}{\operatorname{\mathsf{P}}} 
\newcommand{\E}{\operatorname{\mathsf{E}}}

\newcommand{\R}{\mathbb{R}}

\newcommand{\1}{\mathsf{1}}

\renewcommand{\le}{\leqslant}
\renewcommand{\ge}{\geqslant}

\newcommand{\fl}[1]{\lfloor#1\rfloor}

\newcommand{\X}{{\mathfrak{X}}}

\journal{Statistics and Probability Letters}

\begin{document}

\begin{frontmatter}


\title{Best lower bound on the probability of a binomial exceeding
its expectation}
%



\author{Iosif Pinelis}

\address{Department of Mathematical Sciences\\
Michigan Technological University\\
Houghton, Michigan 49931, USA\\
E-mail: ipinelis@mtu.edu}

\begin{abstract}
Let $X$ be a random variable distributed according to the binomial distribution with parameters $n$ and $p$. 
It is shown that $\P(X>\E X)\ge1/4$ if $1>p\ge c/n$, where $c:=\ln(4/3)$, the best possible constant factor.      
\end{abstract}


\end{frontmatter}



\newcounter{case}
\newenvironment{case}[1][]{\refstepcounter{case}
   \emph{Case~\thecase: #1}}{}

\hspace{50pt}

\section{Summary and discussion}\label{intro}

\begin{theorem}\label{th:}
Let $X=X_{n,p}$ be a random variable (r.v.) with the binomial distribution with parameters $n$ and $p$. Then 
\begin{equation}\label{eq:}
	\P(X>\E X)\ge1/4
\end{equation}
if 
\begin{equation}\label{eq:p}
	1>p\ge c/n,
\end{equation}
where 
\begin{equation}
	c:=\ln(4/3)=0.28768\dots. \label{eq:c}
\end{equation}
Under condition \eqref{eq:p}, the equality in \eqref{eq:} is attained only if $n=2$ and $p=1/2$. 
The constant factor $c$ in \eqref{eq:p} is the best possible. 
\end{theorem}

Complementing Theorem~\ref{th:} is the following simple proposition. 

\begin{proposition}\label{prop:}
If $c/n\ge p\ge0$, then $\P(X>\E X)=1-(1-p)^n\ge\max(1,bn)p$, where $b:= (1-e^{-c})/c=0.86901\ldots$. 
\end{proposition} 

%

A very short proof of Theorem~\ref{th:} will be given in Section~\ref{proofs}. This proof is based on a monotonicity result due to Anderson and Samuels~\cite{anderson-samuels67}, which in turn follows from a more general result due to Hoeffding ~\cite{hoeff56}. 

A bit longer proof of Theorem~\ref{th:}, which may still be of interest, is relegated to the appendix. This second proof is based on a version of the Berry--Esseen bound, which takes care of the main case when $np\ge2$ and $n(1-p)\ge2$, that is, when $2\le\E X\le n-2$. The remaining cases are rather easy to deal with, since all the values of $X$ are in the set $\{0,\dots,n\}$.

Previously it was shown \cite{greenberg-mohri} that, for $X$ as in Theorem~\ref{th:}, one has 
\begin{equation}\label{eq:gm}
	\P(X\ge\E X)>1/4
\end{equation}
if 
\begin{equation}\label{eq:q,gm}
	p>1/n. 
\end{equation}
Theorem~\ref{th:} improves the result of \cite{greenberg-mohri} in two ways at once: 
\begin{enumerate}[(i)]
\setlength\itemsep{1pt}
	\item The (optimal) constant factor $c=0.28768\dots$ in \eqref{eq:p} is better than the corresponding constant factor $1$ in \eqref{eq:q,gm}. (Concerning the strictness of the inequality $\P(X\ge\E X)>1/4$ in \eqref{eq:gm}, here one may recall that the inequality $\P(X>\E X)\ge1/4$ in \eqref{eq:} is strict unless $n=2$ and $p=1/2$ -- in which latter case condition \eqref{eq:q,gm} fails to hold.) 
	\item Instead of the probability $\P(X\ge\E X)$ in \eqref{eq:gm}, we have the (possibly) smaller probability $\P(X>\E X)$ in \eqref{eq:}.
\end{enumerate}

Improvement (i) and the optimality of the constant factor $c$ are illustrated in Figure~\ref{fig:pic}, showing the graphs 
\begin{itemize}
\setlength\itemsep{0pt}
	\item $\{\big(p,\P(X_{n,p}>np)\big)\colon1/n\le p<1\}$ (solid) 
	\item $\{\big(p,\P(X_{n,p}>np)\big)\colon c/n<p\le1/n\}$ (dashed, black)	
	\item $\{\big(p,\P(X_{n,p}>np)\big)\colon 0<p\le c/n\}$ (dashed, gray)
\end{itemize}
for $n=5$. This figure is similar to \cite[Figure~2]{greenberg-mohri}, where the graphs over the interval $(c/n,1/n]$ were dashed, too. 

\begin{figure}[h]
	\centering
		\includegraphics[width=.6\textwidth]{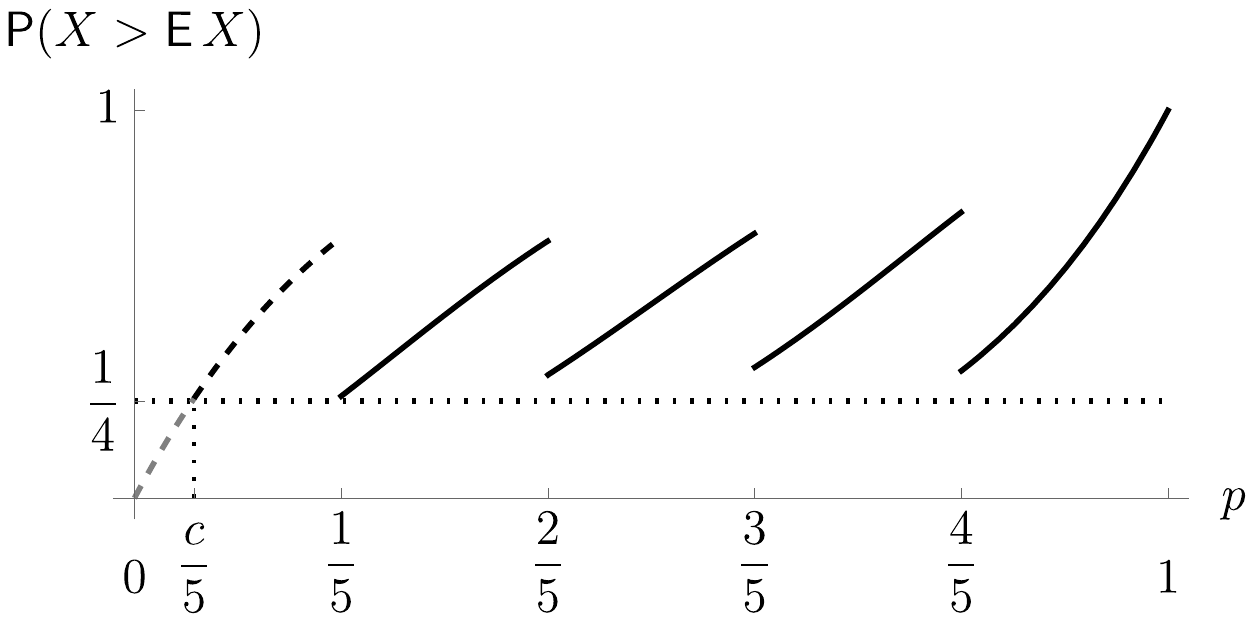}
	\caption{Graphs of $\P(X>\E X)$.}
	\label{fig:pic}
\end{figure}

However, concerning improvement (i), one should note that the case when $c\le np<1$ -- considered 
in the beginning of the proof of Theorem~\ref{th:} -- is comparatively simple. As for improvement (ii), inequality \eqref{eq:} follows from its non-strict counterpart  $\P(X\ge\E X)\ge1/4$ upon noting that $\P(X>\E X)=\P(X_{n,p}>np)$ is right-continuous in $p$ and $\P(X>\E X)=\P(X\ge\E X)\ge1/4$ if $np$ is not an integer. 
%
%
%
%

So, the main distinction of the present note from \cite{greenberg-mohri} is perhaps that each of the two proofs of Theorem~\ref{th:} given here 
%
appears to be 
significantly simpler than the proof in \cite{greenberg-mohri}. 

As noted in \cite{greenberg-mohri}, inequality \eqref{eq:gm} was used several times in the machine learning literature, including \cite{anthony-etal,vapnik98,vapnik06}, 
to bound the probability of the so-called relative deviation of frequencies from the corresponding probabilities for certain classes of events. Such results have applications to the so-called probably-approximately-correct (PAC) models of machine learning; concerning PAC \break 
models, see e.g.\ \cite{DBLP:journals/cacm/Valiant84,DBLP:journals/iandc/Haussler92,KonPin2019,alon2021}. 

%

In \cite{and-jaeger}, the non-strict version, $\P(X\ge\E X)\ge1/4$, of 
inequality ~\eqref{eq:} was obtained, but only for large enough $n$ and $p\ge2/n$. 

In \cite[Lemma~13]{rig-tong}, it was shown that 
\begin{equation}\label{eq:RT}
	\P(X\ge\E X)\ge\min(p,1/4)
\end{equation}
for 
\begin{equation}\label{eq:p RT}
	p\in(0,1/2]. 
\end{equation}
This was used to prove a part of \cite[Proposition~8]{rig-tong}. To state that result, we need to reproduce several definitions from \cite{rig-tong}. Let $(X,Y)$ be a random vector in $\X\times\{-1,1\}$, where $\X$ is a Borel subset of $\R^d$. A classifier is a Borel-measurable map from $\X$ to $\{-1,1\}$. For any classifier $h$, consider the two types of error probabilities, 
\begin{equation*}
	R^-(h):=\P(h(X)\ge0|Y=-1)\quad\text{and}\quad R^+(h):=\P(h(X)<0|Y=1),
\end{equation*}
and also the empirical counterpart 
\begin{equation*}
	\hat R^-(h):=\frac1{n^-}\sum_{i=1}^{n^-}\1(h(X_i^-)\ge0)
\end{equation*}
of $R^-(h)$, where $X_1,\dots,X_{n^-}$ is a (training) iid sample from the conditional distribution of $X$ given $Y=-1$, and
$\1(A)$ denotes the indicator of an assertion $A$ (so that $\1(A)=1$ if $A$ is true and $\1(A)=0$ if $A$ is false). 

The mentioned result in \cite{rig-tong} is as follows: there exist classifiers $h_1$ and $h_2$ and a probability distribution for $(X,Y)$ such that, for any $\al\in(0,1/2]$ and any r.v.\  $\La$ with values in $[0,1]$ such that for the random ``pseudo-classifier'' $h_\La:=\La h_1+(1-\La)h_2$ we have $\hat R^-(h_\La)<\al$, the event that the ``excess type II risk''
\begin{equation*}
	R^+(h_\La)-\min_{\la\in[0,1]\colon R^+(h_\la)}R^+(h_\la)
\end{equation*}
is $\ge\al$ occurs with a probability $P\ge\min(\al,1/4)$. 
 
Using inequality \eqref{eq:} with condition \eqref{eq:p} -- instead of inequality \eqref{eq:RT} with condition \eqref{eq:p RT}, we can replace the conditions $\al\in(0,1/2]$ and $P\ge\min(\al,1/4)$ in the cited result in \cite{rig-tong} by the respective conditions $\al\in[c/n,1]$ and $P\ge1/4$, which will constitute a substantial improvement, in the case when $\al\ge c/n$. For the simpler case of $\al\in(0,c/n]$, an improvement over the result in \cite{rig-tong} can be similarly obtained using Proposition~\ref{prop:}. 


\section{Proofs}\label{proofs}

Here and in what follows, 
\begin{equation}\label{eq:q}
	q:=1-p. 
\end{equation} 

\begin{proof}[Proof of Theorem~\ref{th:}]
If $n=1$, then $$\P(X>\E X)=\P(X>p)
=\P(X=1)=p=np\ge c>1/4,$$ so that \eqref{eq:} holds, with the strict inequality. 

Fix now any natural $n\ge2$. Consider first the case when $c\le np<1$. 
%
%
Then 
\begin{equation}\label{eq:case5}
	\P(X>np)=1 - q^n\ge1 - q^{c/p}=1 - (\tfrac43)^{\frac{\ln(1-p)}p}>1-(\tfrac43)^{-1}
	=\tfrac14,  
\end{equation}
so that $\P(X>np)>\frac14$. Moreover, if $c=\ln\tfrac43$ is replaced here by any $c_1\in(0,c)$, and if $p=c_1/n$ with $n\to\infty$, then $\P(X>np)=1 - q^n=1 -(1-c_1/n)^n\to1-e^{-c_1}<1-e^{-c}=1/4$. 

Therefore, the constant factor $c$ in \eqref{eq:p} cannot be improved and, moreover,  
without loss of generality (wlog)
\begin{equation}\label{eq:np>1}
	np\ge1. 
\end{equation}
So,  
\begin{equation}
	m:=m_n:=\fl{np}+1\in[2,n].   
\end{equation}
Introduce also 
\begin{equation}
	p_j:=p_{n,j}:=(m_n-1)/j=(m-1)/j 
\end{equation}
for $j\in\{m,\dots,n\}$. 
Then 
\begin{equation}\label{eq:proof2}
	\P(X>\E X)=\P(X_{n,p}>np)=
	\P(X_{n,p}\ge m) 
	\ge\P(X_{n,p_n}\ge m).
\end{equation}
The latter inequality, which follows from the (strict) stochastic monotonicity of $X_{n,p}$ in $p$ and the inequality $p\ge p_n$, is strict unless $p=p_n$ (that is, unless $np$ is an integer). 
Next, by part (i) of \cite[Theorem 3]{bin-pois-mono} (which 
immediately follows from the second inequality in \cite[Theorem~2.1]{anderson-samuels67}, again by the stochastic monotonicity of $X_{n,p}$ in $p$), we have 
$\P(X_{j+1,p_{j+1}}\ge m)>\P(X_{j,p_j}\ge m)$ for all $j\in\{m,\dots,n-1\}$. 
So, $\P(X_{n,p_n}\ge m)\ge\P(X_{m,p_m}\ge m)$, and this inequality is strict unless $m=n$. Also, $\P(X_{m,p_m}\ge m)=(1-1/m)^m\ge(1-1/2)^2=1/4$, and $\P(X_{m,p_m}>m)>1/4$ unless $m=2$. 
It follows that $\P(X>\E X)>1/4$ unless $n=m=2$ and $np$ is an integer. 
Thus, in view of \eqref{eq:np>1}, $\P(X>\E X)>1/4$ unless $n=2$ and $p=1/2$. 
That $\P(X>\E X)=1/4$ if $n=2$ and $p=1/2$ is trivial. 
This completes the proof of Theorem~\ref{th:}. 
\end{proof} 

\begin{proof}[Proof of Proposition~\ref{prop:}]
If $c/n\ge p\ge0$, then $\P(X>\E X)=1-(1-p)^n$. Next, $(1-(1-p)^n)/(np)$ is decreasing in $p\in(0,1]$, so that for $p\in(0,c/n]$ we have $(1-(1-p)^n)/(np)\ge(1-(1-c/n)^n)/c\ge (1-e^{-c})/c=b$, so that $1-(1-p)^n\ge bnp$. The inequality $1-(1-p)^n\ge p$ is obvious. 
This completes the proof of Proposition~\ref{prop:}. 
\end{proof}


\bibliographystyle{model2-names}
 
%

\bibliography{C:/Users/ipinelis/Documents/pCloudSync/mtu_pCloud_02-02-17/bib_files/citations04-02-21} 

\appendix
\section*{Appendix}

\begin{proof}[Second proof of Theorem~\ref{th:}]
At least one of the following 
five cases must occur: 

\begin{case}\label{cs:0}
\emph{$np\ge2$ and $nq\ge2$}  
\end{case}
(recall the convention $q:=1-p$ in \eqref{eq:q}). 

\begin{case}\label{cs:4}
\emph{$c\le np<1$ and $n\ge1$, where $c$ is as in \eqref{eq:c}}.  
\end{case}

\begin{case}\label{cs:3}
\emph{$1\le np<2$ and $n\ge3$}.  
\end{case}

\begin{case}\label{cs:1}
\emph{$1<nq\le2$ and $n\ge3$}.  
\end{case}

\begin{case}\label{cs:2}
\emph{$0<nq\le1$ and $n\ge2$}.  
\end{case} 


In particular, note that the cases when either (i) $n=1$ or (ii) $n=2$ and $p<1/2$ are covered by Case~\ref{cs:4}, whereas  
the case when $n=2$ and $p\ge1/2$ is covered by Case~\ref{cs:2}. 

Consider now each of the five listed cases. 

\emph{Case~\ref{cs:0}}: 
%
%
%
The version of the Berry--Esseen bound given in \cite[Theorem~1]{kor-shev12} implies  
\begin{equation*}
	\P(X>\E X)=\P(X>np)\ge \frac12-\ep(n,p),\quad\text{where}\quad
	\ep(n,p):=\frac{c_3}{\sqrt n}\Big(\frac\rho{\si^3}+c_2\Big), 
\end{equation*}
$\rho=p^3q+q^3p$, $\si=\sqrt{pq}$, $c_3:=\frac{33477}{100000}$, $c_2=\frac{429}{1000}$. 

Note that $p^3q/\si^3=p^{3/2}(1-p)^{-1/2}$ is convex in $p$ and, similarly, $q^3p/\si^3$ is convex in $p$, so that $\rho/\si^3$ and $\ep(n,p)$ are convex in $p$. Therefore and in view of 
the Case~\ref{cs:0} conditions $np\ge2$ and $nq\ge2$, we have $\ep(n,p)\le\ep(n,2/n)\break =\ep(n,1-2/n)=:\ep_*(n)$, which is 
a simple algebraic function of $n$. For the derivative $\ep'_*(n)$ of $\ep_*(n)$ in $n$, we see that $\ep'_*(n)n^{5/2}(n-2)^{3/2}$ is a polynomial in $(n-2)^{1/2}$, of degree $5$.  
Therefore, it is easy to see that $\ep_*(n)$ is decreasing in $n\in[4,6]$, increasing in $n\in[7,89]$, and decreasing in $n\in[90,\infty)$. Also, the conditions $np\ge2$ and $nq\ge2$ imply $n=np+nq\ge4$. So, in Case~\ref{cs:0},  $\P(X>np)\ge\frac12-\max(\ep_*(4),\ep_*(89),\ep_*(90))>0.25587>1/4$. 

\emph{Case~\ref{cs:4}}: 
Then, by \eqref{eq:case5}, 
$\P(X>np)>\frac14$. Moreover, it was shown in the paragraph containing \eqref{eq:case5} that 
the constant factor $c$ in \eqref{eq:p} cannot be improved. 

\emph{Case~\ref{cs:3}}: 
Then 
\begin{equation*}
	\P(X>np)=\P(X>1)=1-q^n-nq^{n-1}p,
\end{equation*}
which is 
increasing in $p$, by the stochastic monotonicity of $X_{n,p}$ in $p$. So, wlog $p=1/n$, in which case 
$\P(X>np)=f_3(n):=
1-(2-1/n)(1-1/n)^{n-1}$. The second derivative of $\ln(1-f_3(n))$ is $1/\big((2 n-1)^2 (n-1) n\big)>0$, so that $\ln(1-f_3)$ is convex. Also, $\ln(1-f_3(n))\to\ln(2/e)$. Therefore, $\ln(1-f_3(n))$ is decreasing (in $n\ge3$) and $f_3(n)$ is increasing, from $f_3(3)=7/27>1/4$. Thus, $\P(X>np)>\frac14$ in Case~\ref{cs:3}.

\emph{Case~\ref{cs:1}}: 
Then $n-2\le np<n-1$, $p\ge1-2/n$, and 
\begin{equation*}
	\P(X>np)=\P(X\ge n-1)=f_1(p):=f_1(p,n):=p^n + n p^{n - 1} q,  
\end{equation*}
and $f_1(p)$ is increasing in $p$, by the stochastic monotonicity of $X_{n,p}$ in $p$. Therefore, here 
wlog $p=1-2/n$, and  
\begin{equation*}
	\tilde f_1(n):=f_1(1-2/n,n)=\frac{3 n-2}{n-2}\,(1-2/n)^n.  
\end{equation*}
Letting 
\begin{equation*}
	D\tilde f_1(n):=\tilde f'_1(n)\Big/\frac{(1-2/n)^n \left(3 n-2\right)}{n-2}
	=\ln(1-2/n)+
	\frac{6 n-8}{(n-2) (3 n-2)}, 
\end{equation*}
we have 
\begin{equation*}
	(D\tilde f_1)'(n)=-\frac{4 \left(3 n^2-4 n+4\right)}{(3 n-2)^2 (n-2)^2 n}<0.  
\end{equation*}
So, $D\tilde f_1$ is decreasing. Also, $D\tilde f_1(\infty-)=0$. It follows that $D\tilde f_1>0$ and hence $\tilde f_1$ is increasing, from $\tilde f_1(3)=\frac 7{27}
>\frac14$. Thus, $\P(X>np)>\frac14$ in Case~\ref{cs:1}. 

\emph{Case~\ref{cs:2}}: 
Then $n>np\ge n-1$, $p\ge1-1/n$, and hence 
\begin{equation*}
	\P(X>np)=\P(X=n)=p^n\ge(1-1/n)^n,
\end{equation*}
and $(1-1/n)^n$ is increasing in $n\ge2$, from $(1-1/2)^2=1/4$.  So, $\P(X>np)>\frac14$ in Case~\ref{cs:2} -- except when $n=2$ and $p=1/2$, in which case $\P(X>np)=\frac14$. 

%

This completes the second proof of Theorem~\ref{th:}. 
\end{proof}

\end{document}